\documentclass[a4paper,12pt,intlimits,oneside]{amsart}
\usepackage{amsmath}
\usepackage{amsthm}
\usepackage{latexsym}
\usepackage{amssymb}
\usepackage{xcolor}
\usepackage{dsfont}
\usepackage{mathrsfs}
\usepackage[colorlinks=true]{hyperref}
\numberwithin{equation}{section}
\numberwithin{figure}{section}
\def\expo_#1{{\rm e}^{#1}}

\def\H{{\mathcal H}}

\def\R{\mathbb R}

\def\Z{{\mathbb Z}}

\def\build#1_#2^#3{\mathrel{\mathop{\kern 0pt#1}\limits_{#2}^{#3}}}
\def\td_#1,#2{\mathrel{\mathop{\build\longrightarrow_{#1\rightarrow #2}^{}}}}

\newcommand{\ben}{\begin{equation}}
\newcommand{\een}{\end{equation}}
\newcommand{\beno}{\begin{eqnarray*}}
\newcommand{\eeno}{\end{eqnarray*}}

\newtheorem{theorem}{Theorem}

\newtheorem{proposition}{Proposition}
\newtheorem{lemma}{Lemma}
\newtheorem{remark}{Remark}

\date{January 2023}

\begin{document}
\title[Global Stein Theorem]{Global Stein Theorem on Hardy spaces}
\author{Aline Bonami\address{Aline Bonami, Institut Denis Poisson, D\'epartement de Math\'ematiques, Universit\'e d'Orl\'eans, 45067
Orl\'eans Cedex 2, France}, Sandrine Grellier\address{Sandrine Grellier, Institut Denis Poisson, D\'epartement de Math\'ematiques, Universit\'e d'Orl\'eans, 45067
Orl\'eans Cedex 2, France}
\and Benoit Sehba \address{Benoit Sehba, Department of Mathematics, University of Ghana, PO. Box LG 62 Legon, Accra, Ghana }}
\thanks{ The authors thank the referees for valuable suggestions and comments.}
\maketitle

\begin{abstract}
Let $f$ be an integrable function which has integral $0$ on $\R^n.$ What is the largest condition on $|f|$ that guarantees that $f$ is in the Hardy space $\mathcal H^1(\R^n)?$ When $f$ is compactly supported, it is well-known that the largest condition on $|f|$ is the fact that $|f|\in L \log L(\R^n).$  We consider the same kind of problem here, but without any condition on the support. We do so for $\mathcal H^1(\R^n)$, as well as for the Hardy space $\mathcal H_{\log}(\R^n)$ which appears in the study of pointwise products of functions in $\mathcal H^1(\R^n)$ and in its dual $BMO.$ 

\end{abstract}

\section{Introduction}
The aim of this article is to generalize a well known result by Stein concerning the maximal function. Let us recall it.

\noindent {\bf Theorem (Stein)\cite{stein} }   {\sl Assume that $f$ is a nonnegative integrable function on $\R^n$ that is compactly supported. Then  its maximal function $Mf$ is locally integrable if and only if } 
$$\int  f \ln_+ (f) dx<\infty.$$
{\sl Here, as usual, $\ln_+=\max(\ln,0)$ where $\ln$ denotes the Napierian logarithm. }\\

One may ask about a global version of such a theorem. Of course, it does not make sense without a modification, since the maximal function of such a nonzero function is bounded below by $\frac c{|x|^n}$ at $\infty.$ But it makes sense if the global maximal function is replaced by the local one, that is, if we consider
$$M^{loc}f(x)=\sup_{0<r<1}\frac 1{r^n}{\int}_{|y-x|<r}|f(y)|dy.$$

We will give a necessary and sufficient condition on a non negative $f$ which ensures the integrability of $M^{loc}f$ on $\R^n$, hence a necessary and sufficient condition for a nonnegative $f$ to belong to the local Hardy space $\mathfrak h^1(\R^n).$ 

A nonnegative integrable function cannot be in $\mathcal H^1(\R^n)$ since functions in $\mathcal H^1(\R^n)$ have mean $0.$ But one can ask whether $\displaystyle f-\left(\int f dy\right)\theta$ is in $\mathcal H^1(\R^n).$ Here $\theta$ is a fixed bounded function supported in the unit cube $Q:=(-\frac 12, \frac 12)^n,$ with integral $1.$ The characterization is simple and reminiscent of Stein Theorem.

\noindent {\bf Theorem. }    {\sl Assume $f$ is a nonnegative integrable function on $\R^n.$ Then $\displaystyle  f -\left(\int f dy\right)\theta$ is in $\mathcal H^1(\R^n)$ if and only if } 
\begin{equation}\label{log}
    \int_{\R^n} |f| \left(\ln_+ |f|+ \ln_+ (|x|)\right) dx<\infty.
\end{equation}
{\sl Moreover, condition \eqref{log} is sufficient, regardless of the condition on the sign of $f.$}\\ 
This answers the question raised in the abstract: \\

\noindent {\bf Corollary. }{\sl The vector space of  integrable functions that satisfy \eqref{log} is the largest space $S$ with the following property: if $g$ is in $S$ and $f$ is an integrable function such that $|f|\leq |g|$ and $\displaystyle \int_{\R^n}f dx=0,$   then  $f$ belongs to $\mathcal H^1(\R^n).$}

We will not come back to this corollary later on. So let us deduce it right now from the theorem. First, the theorem implies in particular that the space of functions satisfying condition \eqref{log} has this property. Let us prove that it is the largest. We choose $\theta$  nonnegative, bounded and  supported in $Q.$ We  know that $\theta$ is in $S$ since it satisfies \eqref{log}. Assume that $g$ is nonnegative and belongs to  $S.$ We want to prove that $g$ satisfies \eqref{log}. But, since $S$ is stable under addition (we assumed that it is a vector space),  
$\displaystyle g +\left|\int g dy\right|\theta$ is in $S.$ So the function $\displaystyle g -\left(\int g dy\right)\theta$ belongs to $\mathcal H^1(\R^n),$ and \eqref{log} holds for $g$. 
\smallskip

The second aspect of our work is the generalization to other Hardy spaces. Recall that one equivalent definition of $\mathcal H^1(\R^n)$ is given in terms of maximal functions $\mathcal M_\varphi,$ where $\varphi $ is a smooth function supported in the unit ball with nonzero integral. Following, for instance, \cite{st-book-Real-Variable}, we define, for $f$ an integrable function or, more generally, for $f$ a tempered distribution,
\begin{equation}
  \mathcal M_\varphi f(x)=\sup_{t>0} |\varphi_t*f(x)|, 
\end{equation}
where, as usual, for $t>0$ $\varphi_t=t^{-n}\varphi(t^{-1}\cdot).$
On the other hand, under adequate assumptions on the Musielak function $\Psi:\R^n\times [0,\infty)\to [0,\infty),$  we define the
Musielak-Orlicz-type space $L^\Psi(\R^n)$ as the set of all measurable functions $f$ such that 
$$\int_{\R^n}\Psi\left(x,\frac{|f(x)|}\lambda\right) dx<\infty$$ for some $\lambda>0.$ Typically the class of integrable functions that satisfy \eqref{log} is a Musielak space with $$\Psi (x, t):= t \left(1+\ln_+(t)+\ln_+ |x|\right).$$ This  is a  convex function in the variable $t,$ and $L^{\Psi}$ is a kind of Orlicz space with an Orlicz function that varies with the spatial point $x$. Observe that the dual space of this kind of Musielak spaces is well-known (see \cite{Mus}, \cite{MMO}). For the specific space under consideration, a function $g$ is in its dual  if,  for some $\lambda,$
$$\int\frac {e^{\lambda |g(x)|}}{(1+|x|)^{n+1}}dx<\infty.$$
One recognizes here an easy consequence of  the John-Nirenberg inequality for $BMO$ spaces, which can also be deduced from our results.

Other examples of spaces that are defined in terms of a Musielak functions are Musielak-Orlicz-type Hardy spaces $\mathcal H^\Psi(\R^n).$ These spaces  generalize Hardy spaces $\mathcal H^p(\R^n)$ for $p\leq 1$ (when $\Psi$ does not depend on the variable $x$, they are the spaces introduced by S. Janson \cite{Janson}).

The space $\mathcal H^\Psi(\R^n)$ is the space of tempered distributions $f$ such that $\mathcal M_\varphi f$ belongs to $L^\Psi(\R^n).$ We will consider here the space $\mathcal H_{\log}(\R^n),$ for which $$\Psi(x,t):=\frac{t}{1+\ln_+(t)+\ln_+ |x|}.$$
This space has been introduced in \cite{BGK} in relation with products of functions in $\mathcal H^1(\R^n)$ and $BMO$ and generalized div-curl lemmas. After the seminal paper of Ky \cite{Ky}, which followed this first paper, one can find a large literature on this type of spaces, which are in particular invariant through singular integrals. Here we will not appeal to the deep properties developed there, but use only elementary properties to stick to the generalization of Stein's theorem. We will see that, for $\mathcal H_{\log}(\R^n),$ the class of nonnegative functions for which one has the equivalent of the theorem above is defined by
\begin{equation}\label{log log}
    \int_B |f| \left(1+\ln_+(\ln_+ (|f|))+ \ln_+(\ln_+ (|x|))\right) dx<\infty.
\end{equation}

Note that an integrable function belonging to $\mathcal H_{\log}(\R^n)$ has mean $0.$ Indeed, since $f$ is in $L^1(\R^n)$, we know that $\varphi_t*f$ tends to the constant  $\displaystyle\int f dx$ for $t$ tending to $\infty.$ By  Lebesgue's dominated convergence Theorem, this constant function belongs to $L_{\log}(\R^n):=L^\Psi(\R^n)$, $\Psi$ defined above, which forces the constant to be $0.$ So, as for $\mathcal H^1(\R^n),$ it makes sense to subtract from $f$ the function $\displaystyle\left(\int f dx\right)\theta$ in order to characterize $\mathcal H_{\log}(\R^n).$\\

This problem has already been tackled in \cite{Sola} for $\mathcal H_{\log}(\R^n)$,  for functions with compact support. In other words, the authors of \cite{Sola} established the  analog of Stein's theorem. The interest we found in their paper motivates ours. Here, the new difficulty is the behavior at $\infty.$ A source of inspiration has also been for the first author the remembrance of discussions with François Bouchut in the nineties, and his manuscript \cite{Bouchut}, which already contained \eqref{log} as a sufficient condition.

\smallskip

\noindent {\bf Acknowledgement.} The authors thank François Bouchut for having shared his results with us.\\
\noindent{\bf Notations:} In the following, we will write $A\lesssim B$ (respectively $A\gtrsim B) $ whenever there exists a nonnegative constant $c$ with $A\le cB$ (respectively $A\ge cB$) and $A\simeq B$ whenever $A\lesssim B$ and $A\gtrsim B$. Each time the constants are assumed to be uniform constants, which depend only on the dimension of the space.
%%%%%%%%%%%%%%%%%%%%%%%%%%%%%%%%%%%

\section{The Hardy space $\mathcal H^{1}(\R^n)$ and its  local version $\mathfrak h^1(\R^n)$} 
Let us recall that $\mathcal H^{1}(\R^n)$ is the space of $L^1(\R^n)$ functions $f$ such that, for some smooth function $\varphi$ supported in the unit ball,
\begin{equation}
\|f\|_{\mathcal H^1}:=\|\mathcal M_\varphi(f)\|_1<\infty. 
\end{equation}
The local space $\mathfrak h^1(\R^n)$ is the space of $L^1$ functions such that 
\begin{equation}
\|f\|_{\mathfrak h^1}:=\|\mathcal M_\varphi^{{loc}}(f)\|_1<\infty.
\end{equation}

Here $\mathcal M_\varphi^{{loc}}$ is  the local maximal function, which is defined by taking the supremum on $t<1$ and not on all positive $t.$
\begin{equation}
  \mathcal M_\varphi^{loc} f(x)=\sup_{0<t<1} |\varphi_t*f(x)|, 
\end{equation}

\begin{remark}\label{nonnegative}
It is important to observe that $\mathcal M_\varphi^{loc} (f)\leq CM^{loc}f$ for all $f$. Moreover,  for a nonnegative function $f,$ the local maximal Hardy-Littlewood function $M^{loc}f$ is equivalent to $\mathcal M_\varphi^{loc}$.  \end{remark}
In the following, we will write $$Q:=\left(-\frac 12, \frac 12\right)^n,\, Q_k:=k+Q\text{ for } k\in\Z^n,\text{ and }f_k:=f\chi_{Q_k}.$$ 
Here $\chi_E$ is the characteristic function of the set $E$.\\
We first consider the case of the local Hardy space.   The following statement, which has its own interest, holds.
\begin{theorem}\label{bouchut-loc}  Let $f$ be an integrable function. Then $M^{loc}f$ belongs to $L^1(\R^n)$ if and only if $f=\sum f_k$ satisfies 
\begin{equation}\label{equ-loc}
   \sum_{k\in\Z^n}  \int |f_k(x)|\left(1+\ln_+\left(\frac{|f_k(x)|}{\|f_k\|_1}\right)\right)dx <\infty.
 \end{equation}
In particular, if $f$ satisfies condition \eqref{equ-loc} then $f$ belongs to $\mathfrak h^1(\mathbb R^n)$, the reverse being true if $f$ is nonnegative.
\end{theorem}
One would prefer to have global integrals, but unfortunately this is not the case. Observe that the condition
\begin{equation}\label{ln}
 \int_{\R^n} |f|(1+\ln_+|f|)\,dx <\infty 
\end{equation}
is necessary.
We get a sufficient condition when we add the following one
\begin{equation}\label{amalgam}
\sum_k \|f_k\|_1 \ln_+(\|f_k\|_1^{-1})<\infty.
\end{equation}
This last condition may be interpreted as an amalgam condition: the sequence $(\|f_k\|_1)$ belongs to a kind of $\ell \log \ell$ space of sequences.

\smallskip
\begin{proof}[Proof of the sufficient condition]
 By subadditivity of the maximal function,
$$\|M^{loc}(f)\|_1\leq \sum_k
\|M^{loc}(f_k)\|_1.$$
We will consider each term separately. We first observe that $M^{loc}(f_k)$ has support in $k+2Q$, and, on this cube,
$$M^{loc}(f_k)\leq M(f_k).$$ 
The classical $L\log L$ inequality (see \cite{stein} for instance) written for the normalized function $\frac{f_k}{\|f_k\|_1}$ gives 
 \begin{equation}\label{llogl}
     \int_{k+2Q} Mf_k \,dx \lesssim\|f_k\|_1+\int |f_k|\ln_+\left (\frac{f_k}{\|f_k\|_1}\right) dx.
 \end{equation} 
 We conclude at once.
 
 \end{proof}
   \begin{proof}[Proof of the necessary condition]
As for the sufficient condition, we first prove an estimate for each $f_k.$ We claim that the following estimate holds for $v$ supported in $Q$, 
    \begin{equation}\label{Mloc}
  \int_Q M^{ {loc}} v \,dx\gtrsim \int_Q |v|\ln_+(|v|)\,dx.
  \end{equation}
   The proof is a variant of the proof of \cite{stein} for the local maximal function. We write it for the sake of completeness. Consider the family $\mathcal Q$ of dyadic sub-cubes of $Q.$ The maximal dyadic function is defined on $Q$ by
  \begin{equation}\label{nec-loc}
      M^{d}v(x):=\sup_{x\in R, R\in \mathcal Q}\frac 1{|R|}\int_ R|v(y)|\,dy.
  \end{equation}
  The following classical estimate is the key of the proof. For $v$ with norm $1,$ and $s>1,$
  \begin{equation}\label{SteinIneq0}|\{x, M^d(v)(x)> s\}|\geq \frac 1s \int_{|v|>s} |v|dx
  \end{equation}
  
  Let $\ell(R)$ be the length of the sides of the cube $R.$ The following lemma compares the two maximal functions.
  \begin{lemma}
  Let $v$ be an integrable function supported in $Q$ with norm $1.$ For $x\in Q,$   
   \begin{equation}\label{in-M}
       M^{loc}v(x)\ge c\sup\frac 1{|R|}\int_ R|v(y)|\,dy,
   \end{equation}
   where the supremum is taken on dyadic cubes $R$ containing $x$ and such that $\ell(R)\sqrt n\leq 1.$ Here $c=n^{-n/2}|B(0, 1)|^{-1}.$ 
  \end{lemma}
   \begin{proof}
    Each dyadic cube $R$ is contained in the ball $B(x,\ell(R)\sqrt n),$
   which has the volume $n^{n/2}|B(0, 1)||R|.$
   So
   $$\frac 1{|B(x,\ell(R)\sqrt n)|}\int_{B(x,\ell(R)\sqrt n)}|v|dx>c\frac 1{|R|}\int_R|v|dx.$$
   This ball has radius bounded by $1$ when $\ell(R)\sqrt n\leq 1.$ The inequality between suprema follows at once.
   \end{proof}
   We claim that the supremum defined in the lemma coincides with $M^dv(x)$ when $M^d v(x)>\sqrt n^n.$ Indeed, for larger dyadic cubes,
   we have
   $$\frac 1{|R|}\int_R|v(y)|dy\leq \sqrt n^n\int_Q|v|dy=\sqrt n^n.$$
   So, using  the inequality \eqref{SteinIneq0} for the dyadic maximal function, we get
   \begin{equation}\label{SteinIneq}
       |\{x\in Q, M^{loc}(v)(x)> s\}|\gtrsim \frac 1s \int_{|v|>s/c} |v|dx
   \end{equation} for $s>c\sqrt n^n.$ We integrate both sides from $c\sqrt n^n$ to $\infty$, and find $$\int_{Q}M^{loc}v \,dx \gtrsim\int|v|\ln_+ \left(\frac{|v|}{\sqrt n^n}\right) dx.  $$
   Since the norm of $v$ is $1,$ the estimate  \eqref{Mloc} on $|v| \ln_+(|v|)$ follows at once, using the fact that $\ln_+(t)\leq \ln_+ \left(\frac{t}{\sqrt n^n}\right)+\ln_+(\sqrt n^n).$

 To conclude for the proof, we observe that $M^{loc} f \geq M^{loc} f_k$ on $Q_k.$ Hence, applying the inequality \eqref{Mloc} to each $f_k$ and summing on $k$ allows to get the necessary condition. 
\end{proof}

We now turn to global Hardy space $\mathcal H^1(\R^n)$. 
For $f\in L^1(\R^n),$ let us define 
\begin{equation}\label{cancellation}
    T_\theta f:=f-\left(\int f \,dx\right) \theta.
\end{equation}
Here, $\theta$ is a fixed bounded function supported in the unit cube with $$\displaystyle
\int \theta dx=1.$$

The following theorem may be seen as a global Stein's theorem (see \cite{stein}) and gives as well the necessary condition.
\begin{theorem}\label{bouchut} Let $T_\theta$ be defined as in \eqref{cancellation}. Let $f$ be an integrable function. Then $T_\theta f$ is a function of $\mathcal H^1(\mathbb R^n)$ if  \begin{equation}\label{bouchut-cond}
    \int |f(x)|\left(1+\ln_+|f(x)|+\ln_+|x|\right)dx <\infty.
\end{equation} Moreover, if $f$ is nonnegative 
and $T_\theta f$ is in $\mathcal H^1(\mathbb R^n),$ then condition \eqref{bouchut-cond} holds.
\end{theorem}
In particular, a function $f$ of integral $0$ satisfying \eqref{bouchut-cond} belongs to $\mathcal H^1(\R^n)$. 
\begin{proof}
Even if the sufficient condition has already been established by Bouchut (see the remark below), we will give a slightly different but complete proof of it.

Let us note first that $f$ belongs to $\mathfrak h^1(\R^n).$ Indeed, let us prove that the two sufficient conditions given in \eqref{ln}, \eqref{amalgam} are satisfied. It is straightforward for the first one. We now want to prove that condition \eqref{bouchut-cond}
 implies $\sum_{k\neq 0}\mu_k \ln_+ (\mu_k^{-1})<\infty,$ with $\mu_k=\|f_k\|_1.$  We divide this last sum into two, depending on whether $\mu_k>k^{-(n+1)}$ or not. For the first sum, the inequality comes from the assumption
 $$\int |f(x)|(1+\ln_+|x|)\, dx<\infty$$
 since 
 $$\int |f(x)|(1+\ln_+|x|)\, dx\gtrsim \sum_k (1+\ln_+|k|)\int |f_k(x)|\,dx=\sum_k \mu_k(1+\ln_+|k|) .$$
 For  $\mu_k\leq |k|^{-(n+1)}$, $|k|\ge 1$ using that $x\mapsto x\ln_+(x^{-1})$ is non-decreasing on $(0,e^{-1}]$, the sum of the corresponding terms is bounded by $$\sum_{|k|\ge 1} |k|^{-(n+1)}\ln |k|$$ which is finite.
 Hence, $f$ belongs to the local Hardy space $\mathfrak h^1(\R^n).$ It remains to deal with the non local part.\\
Without loss of generality, we assume $\|f\|_1=1.$ 
As before, we write $f=\sum_k f_k,$ with $f_k=f\chi_{Q_k}$, $k\in\Z^n$. We have as well 
$$ T_\theta f =\sum \left(f_k-\left(\int f_k\, dx\right)\theta\right).$$
We write 
$$T_\theta f_k= f_k-\left(\int f_k\, dx\right)\chi_{Q_k}+ \left(\int f_k\, dx\right)(\chi_{Q_k}-\theta).$$
We will consider separately the two parts.

We first prove the following proposition.
\begin{proposition}\label{close} Let $f_k$ be defined as above and $h:=\sum h_k,$  where $\displaystyle h_k:=f_k-\left(\int f_k \,dx\right)\chi_{Q_k}$ and assume condition  \eqref{equ-loc}. Then $h$ is in  $\mathcal H^1(\R^n).$    

\end{proposition}
\begin{proof}
We prove that the $h_k$'s are in $\mathcal H^1(\mathbb R^n)$ and that 

$$\sum  \|\mathcal M_\varphi  h_k\|_{1}<\infty.$$
This is done by a slight modification of the fact that atoms are in $\mathcal H^1(\mathbb R^n). $
Inside $k+2Q$, we consider the two terms of $h_k$ separately. We conclude directly for the part involving the characteristic function of $Q_k$, while, for $f_k$ we use \eqref{llogl}, which we recall here: 
\begin{equation}\label{llogl-global}
     \int_{k+2Q} M f_k\, dx \lesssim\|f_k\|_1+\int |f_k|\ln_+\left (\frac{f_k}{\|f_k\|_1}\right) dx.
 \end{equation}
 From condition \eqref{equ-loc}, the sum over $k$ of the right hand side is bounded. 
 It follows that
 $$\|\sum_k (M h_k)\chi_{k+2Q}\|_1<\infty.$$
 It remains to prove that 
 $$\|\sum_k ( \mathcal M_\varphi h_k)\chi_{(k+2Q)^c}\|_1<\infty.$$
By the zero-mean of $h_k$, the maximal function $\mathcal M_\varphi h_k $ is bounded by $\|h_k\|_1/|x-k|^{n+1}$ for $x\in \R^n\setminus (k+2Q)$. So
$$\int_{\R^n\setminus (k+2Q)}\mathcal M_\varphi h_k\, dx\lesssim \|h_k\|_1\lesssim \|f_k\|_1.$$
The sum of the corresponding integrals is bounded by some uniform constant. 
The conclusion follows.
\end{proof}

We consider now the remaining part $g:=T_\theta f-h$. Recall that $$g:=\sum_k \int f_k \,dx\left( \chi_{Q_k}-\theta\right).$$ The next lemma gives a sufficient condition for $g$ to be in $\mathcal H^1(\R^n).$ 
\begin{lemma}\label{far-lem} Let $a$ be bounded by $1$, of mean $0,$  and assume that $a$ is supported in $Q_0\cup Q_j,$ with $|j|>2.$ Then 
\begin{equation}\label{far-equ}
    \|a\|_{\mathcal H^1(\R^n)}\lesssim 1+ \ln|j|.
\end{equation} 
\end{lemma}
\begin{proof}
Observe that $|j|^{-1} a$ is an atom of $\mathcal H^1(\R^n),$ and thus has bounded norm. But we need a better estimate. As before, we use the classical inequality
\begin{equation}
    |\mathcal M_\varphi a(x)|\lesssim \frac {|j|}{|x|^{n+1}}
\end{equation}
for $|x|>2|j|.$ For $|x|\le 2|j|$  we use the elementary inequality
\begin{equation}
    |\mathcal M_\varphi a(x)|\lesssim M\chi_{Q_0}(x)+M\chi_{Q_j}(x)\lesssim \frac {1}{1+|x|^{n}}+\frac {1}{1+|x-j|^{n}}.
\end{equation}
  We get
\begin{equation}\label{maj-weight}\|\mathcal M_\varphi a\|_{1}\lesssim 2+\int_{|x|>2|j|} \frac {|j|dx}{|x|^{n+1}}+\int_{1<|x|<2|j|} \frac {dx}{|x|^{n}} +\int_{1<|x-j|<3|j|} \frac {dx}{|x-j|^{n}}.
\end{equation} 
The conclusion follows easily that \eqref{far-equ} holds.
\end{proof}

We have the following proposition.
\begin{proposition}\label{far} Let $\lambda_j$ be a sequence of real numbers indexed by $\mathbb Z^n$ and let $g:=\sum g_j,$  with $g_j:=\lambda_j(\chi_{Q_j}-\theta).$ Then $g$ is in $\mathcal H^1(\R^n)$ if 
\begin{equation}\label{for-far}
\sum_j|\lambda_j|\left(1+\ln_+|j|\right)<\infty.    
\end{equation}
Moreover, if the $\lambda_j'$s are nonnegative, it is a necessary condition for having $g$ in $\mathcal H^1(\R^n).$

\end{proposition}

The sufficiency of the condition comes from Lemma \ref{far-lem}. \\
Let us now assume that the $\lambda_j'$s are nonnegative and prove the necessity of the condition. We choose $\varphi$  such that $ 0\leq\varphi\leq 1,$ and $\varphi=1$ on the ball $B(0, 1/2).$ 

We will give a bound below of $|\varphi_r* g|(x)$ for $|x|>\sqrt n$ and $r=4|x|$. We first prove that $$-\varphi_r* g(x)=\sum_j \lambda_j(\varphi_r*\theta(x)-\varphi_r*\chi_{Q_j}(x))$$ is nonnegative as a sum of nonnegative terms. Indeed, for $|x|>\sqrt n$ and $r=4|x|$, the support of $\theta$ is entirely contained in the set of $y$ for which $\varphi_r(x-y)=1$. So, on one hand
$$\varphi_r*\theta(x)=\frac 1{r^n}.$$
On the other hand, since $1$ is the maximum of $\varphi,$ the other terms $\varphi_r*\chi_{Q_j}(x)$ are bounded by $r^{-n}.$ We have proved our claim on the sign of $-\varphi_r* g(x).$ Now, a bound below is given by $r^{-n}\sum_{j\in J}\lambda_j$ where $J=\{j\in \mathbb{Z}^n\;;\; 
B(x,r)\cap Q_j=\emptyset\}.$ Indeed, for such $j$'s, $\varphi_r*\chi_{Q_j}(x)=0.$
But this set of indices contains all $j$'s such that $|j|\ge 6|x|$ since, under these conditions, 
$$\mbox {dist}(B(x, 4|x|), j)\geq |x|\geq \sqrt n=\mbox{diam}(Q_j).$$

Eventually, 
  $$|\mathcal M_\varphi g(x)|\geq|\varphi_{4|x|}*g(x)|\ge \frac{1}{(4|x|)^n}\sum_{j\geq 6|x|}\lambda_j. $$
 and, integrating over the set $|x|\ge \sqrt n$,
$$\|\mathcal M_\varphi g\|_1\gtrsim \sum_{|j|>6\sqrt n} \lambda_j\ln ( |j|/6).$$
We conclude at once.
\end{proof}

\begin{proof}[End of the proof of Theorem \ref{bouchut}.]

The sufficiency of the condition 
$$\int |f(x)|\left(1+\ln_+|f(x)|+\ln_+(|x|)\right)\,dx<\infty$$
is a consequence of Propositions \ref{close} and \ref{far}. Indeed, we already mentioned that condition \eqref{equ-loc} holds. Hence, what we called $h$ in Proposition \ref{close} belongs to $\mathcal H^1(\R^n)$. To prove that $g=f-h$ is in $\mathcal H^1(\R^n)$, one has to prove estimate \eqref{for-far} with $\displaystyle\lambda_j=\int f_j.$ Since $|x|\simeq |j|$ on $Q_j$, it follows easily from the discretization of the integral 
$\int |f(x)|\left(1+\ln_+(|x|)\right)dx.$
\medskip

We now prove the necessity for a nonnegative function $f$ with $T_\theta f\in\mathcal H^1(\R^n)$. As $\mathcal H^1(\R^n)$ is contained in $\mathfrak h^1(\R^n)$, $T_\theta f$ belongs to $\mathfrak h^1(\R^n)$. Since $\theta$ is also in $\mathfrak h^1(\R^n),$ the function $f$ itself is in $\mathfrak h^1(\R^n).$ So the condition \eqref{equ-loc} is satisfied.
By Proposition  \ref{close},  $h$ is in $\mathcal H^1(\R^n).$ So $g=f-h$ is in $\mathcal H^1(\R^n).$ By Proposition \ref{far}, this implies $$\sum_{|k|>1}\mu_k (1+\ln_+ |k|)<\infty$$ with $\displaystyle\mu_k=\int f_k.$
As before, this last condition reads 
$$\int |f(x)|(1+\ln_+(|x|))\, dx<\infty.$$
Eventually, combined with the estimate \eqref{ln},
one gets 
$$\int |f(x)|(1+\ln_+|x|+\ln_+(|f(x)|)\, dx<\infty.$$

\end{proof}
\smallskip

\noindent {\bf Remark.} The precise statement of Bouchut in  \cite{Bouchut} is the following. 

{\sl Let $T$ be a singular integral operator given by the convolution by $K.$ Then the operator $R_K, $ defined  
by
$$R_K(f)(x)=T(f)(x)-\left(\int f\,dy\right) K(x)\chi_{|x|>1}$$
maps the space of functions that satisfy \eqref{log} into $L^1(\R^n).$}

This may be obtained as a consequence of Theorem \ref{bouchut}. Indeed, let $f$ satisfy \eqref{log}. We know that $T_\theta f$ is in $\mathcal H^1(\R^n).$ Since singular integral operators map $\mathcal H^1(\R^n)$ into itself and since the integral of $T_\theta f$ is zero, $R_K(T_\theta f)$ is integrable. So it is sufficient to see that $R_K(\theta)$ is also in $L^1(\R^n).$  Because of $L^p$ estimates for singular integral operators, $T(\theta)$ is locally integrable. The integrability on $\{|x|>2\}$  follows from the estimate
$$\int_{|x|>2}\int_{|y|<1}|K(x-y)-K(x)|\theta (y)dy dx<\infty$$
since the kernel $K,$ which is the kernel of a singular integral, satisfies
$$\int_{|x|>2}|K(x-y)-K(x)|dx<\infty$$
for $|y|<1.$

Conversely, by using the characterization of $\mathcal H^1(\R^n)$ through Riesz transforms, it is easy to see that the statement of Bouchut implies the sufficient condition in Theorem \ref{bouchut}. 

\section{The local Hardy space $\mathfrak h_{\log}$} 

We first consider the local Hardy space $\mathfrak h_{\log}(\mathbb R^n)$. 
The following result gives the analogue of Theorem \ref{bouchut-loc}. 
\begin{theorem}\label{bouchut-loc_log}  Let $f$ be an integrable function. Then $M^{loc}f$ is in $L_{\log}(\mathbb R^n)$ if and only if  \begin{equation}\label{eq-loglog}
    \int |f(x)|\left(1+\ln_+\frac{\ln(e+|f(x)|)}{\ln(e+|x|)}\right)dx <\infty.
\end{equation}  
As a consequence, if $f$ satisfies this condition then $f$ is in $\mathfrak h_{\log}(\mathbb R^n)$. 
Conversely, if $f$ is nonnegative and $f$ is in $\mathfrak h_{\log}(\mathbb R^n),$ then condition \eqref{eq-loglog} holds.
\end{theorem}
\begin{proof}
 It will be convenient to take an equivalent function for $\Psi,$ which leads to the same spaces. Namely, let $$\overline \Psi(x,t):=\frac t{\ln (e+t)+\ln( e+|x|)}.$$ 
We omit the bar from now on. The function $\Psi$ is non-decreasing in $t$. An explicit computation gives
\begin{equation}\frac 12 \frac{\Psi(x,t) }{t} \leq \frac d{dt} \Psi (x, t)\leq \frac{\Psi(x,t) }{t}.
\end{equation} 
It is also doubling:
$$\Psi(x,2t)\lesssim \Psi (x, t).$$

We first establish the following lemma, which may be seen as a precised version of the corresponding result in \cite{Sola}.
\begin{lemma}\label{MaxQ}
Let $f_k$ be supported in $k+Q$ with $\mu_k:=\Vert f_k\Vert_{1}\le 1$. One has
\begin{align}\label{estimQ}
\int_{k+2Q} \Psi(x,M(f_k))\,dx\lesssim \frac{\mu_k}{\ln(e+|k|)}&+\frac{\mu_k}{\ln(e+|k|)}\ln_+\left(\frac 1{\mu_k}\right)\\&+\int_{k+Q}|f_k(x)|\left(1+\ln_+\frac{\ln(e+|f_k(x)|)}{\ln(e+|k|)}\right)\,dx .\nonumber
\end{align}
\end{lemma}
\begin{proof}
We have
\begin{equation}
    \Psi(x,t)\simeq \Psi(|k|, t)\qquad\text{for } x \in k+2Q.
\end{equation}
We write, as it is classical (see for instance \cite{stein}, Chapter 1) 
\begin{align*}
\int_{k+2Q} \Psi(k,M(f_k)(x))\, dx &\simeq\int _0^\infty \Psi'(k, t)|\{x\in k+2Q \;;\;Mf_k(x)>t\}| dt\\
\lesssim \Psi (k,\Vert f_k\Vert_1)+&\int_{\Vert f_k\Vert_{1}}^{\infty}\frac{\Psi(k,t)}t\, |\{x\in k+2Q \;;\;Mf_k(x)>t\}| \,dt.
\end{align*}
Since $\displaystyle\Psi(x, t)\leq \frac t{\ln (e+|x|)},$ the first term is bounded by $\displaystyle\frac{\mu_k}{\ln(e+|k|)}.$

Using the maximal theorem, one has the inequality
$$|\{x\in k+2Q \;;\;Mf_k(x)>t\}|\lesssim \frac 1 t \int_{|f_k|\ge t/2} | f_k (x)|\,dx.$$
So, using this last inequality and exchanging the integrals,  we get $$\int_{k+2Q} \Psi(k,M(f_k)(x))\, dx \lesssim \frac{\mu_k}{\ln(e+|k|)}+\int_{2|f_k(x)|>\mu_k} |f_k(x)|\int_{\mu_k}^{2|f_k(x)|}\frac{\Psi(k, t)}{t^2}\,dt\,dx.$$

We cut the last integral into the integral below $1$ and above $1.$ Since $\Psi$ is doubling, it is straightforward to see that such integrals are also doubling, namely
\begin{equation}
    \int_1^{2s}\frac{\Psi(x,s)}{s^2}\,ds\lesssim  \int_1^{s}\frac{\Psi(x,s)}{s^2}\,ds
\end{equation}
for $s>2.$ We will use the following properties satisfied by $\Psi.$
\begin{enumerate}
   
    \item for $0<t<1$,$$\int_t^1\frac{\Psi(x,s)}{s^2}\,ds\simeq \frac{\ln_+ (1/t)}{\ln(e+|x|)},$$
    \item for $t>1$, $$\ln_+\frac{\ln(e+t)}{\ln(e+|x|)}\lesssim\int_1^t\frac{\Psi(x,s)}{s^2}\,ds \lesssim 1+\ln_+\frac{\ln(e+t)}{\ln(e+|x|)}.$$
\end{enumerate}
Indeed,  we write, on one hand
for $0<t<1$,$$\int_t^1\frac{\Psi(x,s)}{s^2}\,ds=\int_t^1\frac 1{s(\ln(e+s)+\ln(e+|x|))}\,ds\simeq \int_t^1\frac 1{s\ln(e+|x|)}\,ds$$
on the other hand, for $t>1$,
\begin{eqnarray*}
\int_1^t\frac{\Psi(x,s)}{s^2}\,ds&=&\int_1^t\frac 1{s(\ln(e+s)+\ln(e+|x|))}\,ds\\
&\simeq& \int_1^t\frac 1{(s+e)(\ln(e+s)+\ln(e+|x|))}\,ds\le\ln\left(1+\frac{\ln(e+t)}{\ln(e+|x|)}\right).
\end{eqnarray*}
Using these properties, we get the inequalities
$$\int_{\mu_k}^1 \frac{\Psi(k,t)}{t^2}\,dt\lesssim \frac{1}{\ln(e+|k|)}\ln_+\left(\frac 1{\mu_k}\right),$$
and 
$$\int_{1}^{2|f_k(x)|} \frac{\Psi(k,t)}{t^2}\,dt\lesssim 1+ \ln_+\left(\frac{\ln(e+|f_k(x)|)}{\ln(e+|k|)}\right).$$
It follows that
\begin{eqnarray*}\int_{k+2Q} |f_k(x)|\int_{\mu_k}^{2f_k(x)}\frac{\Psi(k, t)}{t^2}\,dt\,dx &\lesssim&\frac{\mu_k}{\ln(e+|k|)}+ \frac{\mu_k}{\ln(e+|k|)}\ln_+\left(\frac 1{\mu_k}\right)+\\ & & \int_{k+Q}|f_k(x)|\left(\ln_+\left(\frac{\ln(e+|f_k(x)|)}{\ln(e+|k|)}\right)\right)\,dx.
\end{eqnarray*}

This ends the proof of the Lemma.
\end{proof}

Let us prove the local theorem \ref{bouchut-loc_log}.
Let $f$ be an integrable function. Without loss of generality, we can assume $\Vert f\Vert_{1}=1$. We first prove the sufficiency of condition \eqref{eq-loglog} to have $M^{loc}(f)\in L_{\log}(\R^n)$. We have to estimate
$$\int_{\R^n} \Psi(x,M^{loc}(f)(x))\,dx. $$  We write $f$ as $\sum f_k$ with $f_k$ supported in $k+ Q$, $k\in\Z$, $\mu_k:=\Vert f_k\Vert_{1}$. At this point, we recall that $L_{\log}$ is only a quasi Banach space, so that we need to be careful. We will use the following property, valid for nonnegative functions $g_j,$ that 
\begin{equation}
    \int_{\R^n} \Psi(x,\sum g_j (x))\,dx\lesssim \sum \int_{\R^n} \Psi(x,g_j(x))\,dx.  
\end{equation}
This is an easy consequence of the fact that the function $\displaystyle t\mapsto \frac{\Psi(x,t)}t$ is nonincreasing.
We apply Lemma \ref{MaxQ} to $f_k$ to get,
\begin{eqnarray*}
&&\int\Psi(x,M^{loc}(f_k)) \,dx\le \int_{k+2Q} \Psi(x,M(f_k)) \,dx\\
&\lesssim& \frac{\mu_k}{\ln(e+|k|)}\left(1+\ln_+\left(\frac 1{\mu_k}\right)\right)+\int_{k+Q}|f_k(x)|\ln_+\frac{\ln(e+|f_k(x)|)}{\ln(e+|k|)}\,dx
\end{eqnarray*}
It remains to sum over $k$. The sums corresponding to the two first terms are bounded by the norm of $f$ in $L^1$ (we cut the second sum into two parts as in the preceding proof by comparing $\mu_k$ to $|k|^{-(n+1)}$). Eventually, using that $|x|\simeq |k|$ on $k+Q$,  
$$
\int_{\R^n} \Psi(x,M^{loc}(f)(x))\,dx \lesssim \Vert f\Vert_{1}+\int |f(x)|\ln_+\frac{\ln(e+|f(x)|)}{\ln(e+|x|)}\, dx.
$$
The conclusion follows for the sufficient condition. For the necessary condition, we use Stein Inequality \eqref{SteinIneq} as for the characterization of $\mathfrak h^1(\R^n)$. It gives
\begin{eqnarray*}
\int\Psi(x,M^{loc}(f_k))\,dx&\gtrsim& \int_{(k+Q)\cap |f_k(x)|>1} |f_k(x)|\int_1^{c^{-1}|f_k(x)|}\frac{\Psi (|k|,s) }{s^2}\,ds\,dx\\
&\gtrsim &\int_{k+Q}|f_k(x)|\ln_+\frac{\ln_+(|f_k(x)|)}{\ln(e+|k|)}\,dx
\end{eqnarray*}
Summing on $k$ gives the result as before.
Hence, we proved that $M^{loc}f$ belongs to $L_{\log}(\R^n)$ if and only if condition \eqref{eq-loglog} holds. The result on $\mathfrak h_{\log}(\R^n)$ is an easy consequence of the remark we did before that $\mathcal M_\varphi^{loc} (f)\leq CM^{loc}f$ while, when $\varphi$ and $f$ are nonnegative, a reverse inequality is also valid.

\end{proof}

Before leaving the local estimates, let us answer a natural question. What can we say for functions that are only locally integrable? We will see that there exist locally integrable functions, which are not integrable but belong to $\mathfrak h_{\log}(\R^n).$ Thanks to the differentiation Lebesgue Theorem $|f|\leq M^{loc}f$, hence it is clear that $f\in L_{\log}(\R^n)$ is a necessary condition for a nonnegative function $f$ to belong to $\mathfrak h_{\log}(\R^n).$ The following lemma goes into this direction. \begin{lemma}\label{loc-log-lem}
Assume that $f$ is nonnegative and locally integrable. If $$\int \Psi(x, M^{loc}f) \,dx<\infty,$$
then the function $f$ belongs to the weighted space $L^1(\ln (e+|x|))^{-1}dx).$\\
In particular, if $\displaystyle\int \Psi(x, M^{loc}f) dx=1,$ then 
$$\int_{\R^n} \frac{|f(x)|}{\ln(e+|x|)}\, dx\lesssim 1.$$
\end{lemma}
\begin{proof}
If we note as before $\displaystyle\mu_k:=\int f_k\, dx,$ it is sufficient to prove that 
\begin{equation}\label{muk}
    \sum_k \frac{\mu_k}{\ln(e+|k|)}<\infty.
\end{equation}
Let $N$ be the smallest integer such that $2^{-N}\sqrt n\leq 1.$  Using the comparison between $M^{loc}$ and $M^d$ given in \eqref{in-M}, if we consider any sub-cube $R$ of $k+Q$ such that $\ell(R)=2^{-N},$ then 
$\displaystyle M^{loc} f_k(x)\gtrsim \int_R f_k\, dy$ for $x\in R.$ We call $\mathcal Q_{k,N}$ the family of these $2^{Nn}$ subcubes. It follows that
\begin{align*}
    \Psi(k, \mu_k)&\leq \sum_{R\in \mathcal Q_{k,N}} \Psi(k,\int_R f_k(y)\, dy )\\ &\lesssim \sum_{R\in \mathcal Q_{k,N}}\int_R  \Psi(k, M^{loc}f_k)\,dx=\int_{k+Q}  \Psi(k, M^{loc}f_k)\,dx.
\end{align*}
So
$$\sum_k\Psi(k, \mu_k)\lesssim \sum_k\int_{k+Q} \Psi(x, M^{loc}f_k) \, dx\leq \int \Psi(x, M^{loc}f)\, dx.$$
Next, one verifies that $\displaystyle\frac{\mu_k}{\ln(e+|k|)}\leq 2\Psi(k, \mu_k)$
unless $\mu_k$ is larger than $|k|,$ which implies, by monotonicity, that $\Psi(k, |k|)\leq 2\Psi(k, \mu_k).$  Since the sequence $(\Psi(k, \mu_k))_k$ tends to $0,$ this happens only for a finite number of terms, which proves the convergence of the series \eqref{muk}.\\
Assume now that 
 $\displaystyle\int \Psi(x, M^{loc}f)\, dx=1,$
 which implies that 
 $$\sum_k\Psi(k, \mu_k)\lesssim 1.$$
 The sum $\displaystyle\frac{\mu_k}{\ln(e+|k|)},$ when restricted to those $k$ for which $\mu_k\leq |k|$ is bounded by some uniform constant.
 The number of the other terms is bounded by a uniform constant. Moreover, for any of them 
 $\displaystyle\frac{\mu_k}{\ln(e+\mu_k)}\le 2\Psi(|k|,\mu_k)\lesssim 1.$ It follows that each of these $\mu_k$ is bounded by some uniform constant. The estimate of the integral follows at once.
\end{proof}
The  function 
$$ f(x):=(1+|x|)^{-n}\left(\ln(e+|x|)\right)^{-1},$$
which varies very slowly, gives an example of non integrable  function in $\mathfrak h_{\log}(\R^n).$ The following proposition revisits Theorem \ref{bouchut-loc_log} without assuming a priori that $f$ is integrable. It is much more technical. The notations for $f_k$ and $\mu_k$ are the same.
\begin{proposition}
Let $f$ be locally integrable.  Then we have following inequality.
 \begin{align}
   \int \Psi(x, M^{loc}f)\,dx&\lesssim  \sum_{k}\frac{\mu_k}{\ln(e+|k|)}+\int_{\R^n}|f(x)|\ln_+\left(1+\frac{\ln(e+|f(x)|)}{\ln(e+|x|)}\right)\,dx\nonumber\\ & +\sum_k\int_{\R^n}|f_k(x)|\frac{\ln_+\left(\frac {\min(|f_k(x)|,|k|)}{\mu_k}\right)}{\ln(e+|k|)}\, dx .  \end{align}
   Moreover, the finiteness of the right hand side  is a necessary condition for $M^{loc} f$ to be  in $L_{\log}(\R^n).$
\end{proposition}
\begin{remark} If $f$ is such that $|f_k|\lesssim \mu_k$ and $|f(x)|\lesssim |x|$ then this condition reads $$\int \frac{|f(x)|}{\ln (e+|x|)}\, dx<\infty.$$
This is the case of the example given above
$$ f(x):=(1+|x|)^{-n}\left(\ln(e+|x|)\right)^{-1}.$$
\end{remark}
\begin{proof}

We will only sketch the proof since it follows the same lines as the proof of  Theorem \ref{bouchut-loc_log}. When revisiting this proof, we cut now the integral 
$$\int_{\mu_k}^{2|f_k(x)|}\frac{\Psi (k,s)}{s^2}\,ds=\int_{\mu_k}^{\min (|k|, 2|f_k(x)|)}\frac{\Psi (k,s)}{s^2}\,ds+\chi_{\{2|f_k(x)|>|k|\}}\int_{|k|}^{2|f_k(x)|}\frac{\Psi (k,s)}{s^2}\,ds.$$
The second integral is treated as previously. For the first one, we use
$$\frac{\Psi (k,s)}{s^2}\simeq \frac 1{s\ln(e+|k|)},$$
which leads to the last term. 
\end{proof}

\section{The Hardy space $\mathcal H_{\log}$}
We finally give conditions for $\mathcal H_{\log}$ and prove the following theorem.

\begin{theorem}\label{bouchut_log} Let $T_\theta$ be defined as in \eqref{cancellation}. Let $f$ be an integrable function. Then $T_\theta f$ is a function of $\mathcal H_{\log}(\mathbb R^n)$ if  \begin{equation}\label{bouchut-cond_log}
    \int |f(x)|\left(1+\ln_+\ln(e+|f(x)|)+\ln_+\ln(e+|x|)\right)dx <\infty.
\end{equation} Moreover, if $f$ is nonnegative and $T_\theta f$ is in $\mathcal H_{\log}(\mathbb R^n),$ then \eqref{bouchut-cond_log} holds.
\end{theorem}
As for $\mathcal H^1(\R^n)$, a function $f$ of integral $0$ and satisfying the condition \eqref{bouchut-cond_log} belongs to $\mathcal H_{\log}(\R^n)$.
\begin{proof}
 We write as before $T_\theta f=g+h.$ To prove that the function $h$ belongs to $\mathcal H_{\log}(\mathbb R^n)$, we do as in the proof of Theorem \ref{bouchut} and look first at the $L_{\log}$ norm of each $Mh_k$ on $k+2Q$. On this set, we consider the two terms in $h_k$ separately. We conclude directly for the characteristic function whose sum gives a $L^1$ term, while, for $f_k$ we use the estimate \eqref{estimQ}. It remains to look at the integral of $Mh_k$ outside $k+2Q$. We use the fact that $\Psi(x, t)\leq t,$ and may consider $L^1$ norms as in Proposition \ref{close}. The proof is identical. 
 
 We finally want to estimate $g.$ In this case, Proposition \ref{far} is replaced by the following one, which may be seen as its generalization.

\begin{proposition}\label{far2} Let $\omega$ be a radial nonincreasing positive function  on $\R^n,$  which satisfies
\begin{equation}
    \int_{\R^n} \frac{\omega(x)}{(1+|x|)^{n+1}}dx <\infty.
\end{equation}
Let $(\lambda_j)_{j\in \Z^n}\in \ell^1(\Z^n)$ and let  $g:=\sum g_j,$  with $g_j:=\lambda_j(\chi_{Q_j}-\theta).$ Then $g$ is in $\mathcal H^1_{\omega}(\R^n)$ if \begin{equation}\label{for-far2}
\sum_j\lambda_j\Omega(|j|)<\infty    
\end{equation}
where $\Omega$ is the function 
$$\Omega(R):= \omega(Q_0)+ \int_{1<|y|<R}\frac{\omega(y)}{|y|^n}\, dy.$$
Moreover, if the $\lambda_j'$s are nonnegative, it is a necessary condition for having $g$ in $\mathcal H^1_{\omega}(\R^n).$
\end{proposition}
\begin{proof}
We first note that the function $\Omega$ is doubling, that is, there exists $C$ such that $\Omega (2R)\leq C\Omega(R).$ Next we adapt Lemma \ref{far-lem}. Inequality \eqref{maj-weight} for $a_j=\chi_{Q_j}-\theta$ is replaced by
\begin{align}\nonumber
    \int |\mathcal M_\varphi a_j| \omega(x) dx\leq \omega(Q_0)&+\omega(Q_j)+\int_{|x|>2|j|} \frac {C|j|\omega(x)}{|x|^{n+1}}\,dx\\&+\int_{1<|x|<2|j|} \frac {C\omega(x)}{|x|^{n}}\,dx +\int_{1<|x-j|<3|j|} \frac {C\omega(x)}{|x-j|^{n}}\,dx.\label{maj-weight2}
\end{align}
From the facts that $\omega$ is radial and nonincreasing, it follows that $\omega (Q_0)\geq \omega (Q_j) $ and 
$$ \int_{1<|x-j|<3|j|} \frac {\omega(x)}{|x-j|^{n}}\,dx\leq \int_{1<|x|<3|j|} \frac {\omega(x)}{|x|^{n}}\,dx\lesssim\Omega (|j|).$$
It remains to consider the term $$\int_{|x|>2|j|} \frac {|j|\omega(x)}{|x|^{n+1}}\,dx=\int_{|x|>2} \frac {\omega(|j| x)}{|x|^{n+1}}\,dx\lesssim 1.$$
This gives the sufficient condition. 

The necessary condition is obtained as in the proof of Proposition \ref{far}. Indeed, one has as before, for $|x|\ge \sqrt n$,
$$|\mathcal M_\varphi g(x)|\geq \frac{1}{(4|x|)^n}\sum_{j\geq 6|x|}\lambda_j. $$
 Integrating against $\omega$ over the set $ |x|\ge \sqrt n$, we get
 $$\sum_j \lambda_j \int_{1< |x|<| j|}\frac{\omega(x)}{|x|^n}dx<\infty.$$ It gives the necessary condition. 

\end{proof}

In particular, when $\omega(x)=\ln( e+|x|)^{-1},$ we have $\Omega(x)\simeq 1+\ln(\ln (e+|x|),$ which allows to conclude for the sufficient condition in Theorem \ref{bouchut_log}. 

It remains to prove the necessary condition in Theorem \ref{bouchut_log}. Let us assume that $f$ is a nonnegative function such that $T_\theta f$ belongs to $\mathcal H_{\log}(\R^n)$. Then  $T_\theta f\in \mathfrak h_{\log}(\mathbb R^n)$, and as $\theta\in \mathfrak h^{1}(\mathbb R^n)\subset \mathfrak h_{\log}(\mathbb R^n)$, we conclude that $f$ itself belongs to $\mathfrak h_{\log}(\mathbb R^n)$. Hence by Theorem \ref{bouchut-loc_log}, $f$ satisfies (\ref{eq-loglog}). If we cut $f$ into $g+h$ as before, Lemma \ref{MaxQ} implies that the function $h$ is in $\mathcal H_{\log}(\R^n)$ (apply estimate \eqref{estimQ} to each $h_k$). Hence, $g$ is also in $\mathcal H_{\log}(\R^n).$ To conclude, we claim that we can work on weighted inequalities since, as $|g|\leq 2\|f\|_1=2,$ we have
$$ \int_{\R^n}\Psi(x, \mathcal M_\varphi g)\, dx\simeq\int_{\R^n} \mathcal M_\varphi g(x)\; \frac{dx}{\ln(e+|x|)}.$$
Hence, $g$ belongs to $\mathcal H^1_\omega(\R^n)$ where 
$\omega (x)= \ln(e+|x|)^{-1}.$ Then, it suffices to apply by Proposition 
 \ref{far2}, with $\displaystyle\lambda_k=\int f_k\,dx$ to get $$\sum_k \lambda_k (1+\ln(\ln (e+|k|))<\infty$$ which is equivalent to 
 
$$\int f(x)(1+\ln\ln(e+|x|))dx<\infty.$$ 
The conclusion follows.
\end{proof}

\section{Concluding remarks} 

For the  estimates in $\mathfrak h_{\log}(\R^n)$ we have allowed $f$ not to be integrable. This could also be done for the space $\mathcal H_{\log}(\R^n),$ even if only $\displaystyle\left(\int f_k\, dx\right)\theta$ makes sense for $f$ locally integrable. Other possible generalizations concern the weighted spaces. Part of this study can be generalized to other Hardy spaces of Musielak type. But formulas seem to be much more complicated and have less interest.  
\bibliographystyle{plain}

\end{document}